\documentclass[11pt,a4paper]{amsart}
\pdfoutput=1

\usepackage[utf8]{inputenc}

\usepackage{hyperref,graphicx,amssymb}

\usepackage{amsfonts}
\usepackage{amssymb}
\usepackage{amsmath}
\usepackage{amsthm}
\usepackage{mathtools}
\usepackage{tikz}

\newtheorem*{theorem*}{Theorem}
\newtheorem*{lemma*}{Lemma}

\begin{document}

\title[$PSp_6(2)$ as Galois group over $\mathbb{Q}(t)$ of degree 36]{A one-parameter family of degree 36 polynomials with $PSp_6(2)$ as Galois group over $\mathbb{Q}(t)$  }

\author{Dominik Barth}
\author{Andreas Wenz}

\address{Institute of Mathematics\\ University of Würzburg \\ Emil-Fischer-Straße 30 \\ 97074 Würzburg, Germany}
\email{dominik.barth@mathematik.uni-wuerzburg.de}
\email{andreas.wenz@mathematik.uni-wuerzburg.de}

\begin{abstract}
We present a one-parameter family of degree $36$ polynomials with the symplectic $2$-transitive group $PSp_6(2)$ as Galois group over $\mathbb{Q}(t)$.
\end{abstract}
\maketitle

In the following we will refer to well known facts about covers of the Riemann sphere $\mathbb{P}^1\mathbb{C}$ and Hurwitz spaces appearing in e.g. \cite{FV}, \cite{MM} and \cite{Voe}. 

Let $\vec{\mathcal{C}}=(\mathcal{C}_1,\mathcal{C}_2,\mathcal{C}_2,\mathcal{C}_3)$ be the class vector of the group $PSp_6(2)\xhookrightarrow{} S_{36}$, where the conjugacy classes $\{\mathcal{C}_i\}_{i=1,2,3}$ are unique of type $(3^{12})$, $(1^{12}.2^{12})$, and $(1^6.2.4^7)$, and $\mathcal{H}$ the associated $(2,3)$-symmetrized Hurwitz curve\footnote{$\mathcal{H}$ can be interpreted as the family of four-branch-point covers of $\mathbb{P}^1\mathbb{C}$ ramified over $0$, $1\pm \sqrt{\lambda}$, $\infty$ with ramification data $\vec{\mathcal{C}}$.}. The corresponding straight inner Nielsen class is of length $2$ and forms a single orbit under the braid group action. Therefore, the branch point reference map $\mathcal{H} \to \mathbb{P}^1\mathbb{C}$ is of degree 2, and ramified over two rational points.
Combining this observation with the rationality of all classes in $\vec{\mathcal{C}}$ yields that $\mathcal{H}$ is a rational genus-$0$ curve over $\mathbb{Q}$.
This implies that $PSp_6(2)$ occurs as a Galois group over $\mathbb{Q}(a,t)$ where the ramification with respect to $t$ is described by $\vec{\mathcal{C}}$.

In order to obtain an explicit polynomial with $PSp_6(2)$ as Galois group we follow the method described in \cite{BWJ} by computing a four-branch point cover $f$ corresponding to $\vec{\mathcal{C}}$. Assume $f$ has the ramification locus consisting of $0,1,-1,\infty$, then $f^2$ turns out to be a Belyi map ramified over $0$, $1$, $\infty$, and its (transitive) monodromy group is contained in the wreath product $PSp_6(2) \wr C_2 \xhookrightarrow{}S_{72}$. The corresponding ramification has to be of type $(6^{12})$, $(1^{24}.2^{24})$ and  $(2^6. 4. 8^7 )$. Now, $PSp_6(2) \wr C_2$ contains exactly one triple $(x,y,z)$ (up to simultaneous conjugation) which satisfies $xyz =1$ and the above conditions describing the monodromy of $f^2$. It is given by
\begin{align*}
x=\;& 
(1, 37, 16, 70, 23, 59)(2, 51, 13, 43, 7, 49)(3, 39, 32, 71, 28, 46) \\
&(4, 66, 26, 72, 34, 52)(5, 42, 31, 67, 10, 64)(6, 41, 22, 69, 29, 65)\\
&(8, 56, 12, 48, 21, 54)(9, 45, 30, 40, 14, 50)(11, 47, 33, 58, 18, 57)\\
&(15, 38, 19, 60, 24, 55)(17, 53, 20, 44, 35, 68)(25, 61, 27, 63, 36, 62),
\end{align*}
and
\begin{align*}
z=\;&  (1, 71, 35, 40, 4, 52, 16, 37)(2, 46, 10, 67, 31, 65, 29, 38)\\
&(3, 49, 13, 56, 20, 68, 32, 39)(5, 64, 28, 59, 23, 72, 36, 41)(6, 42)\\
&(7, 43)(8, 54, 18, 66, 30, 50, 14, 44)(9, 45)(11, 57, 21, 47)\\
&(12, 51, 15, 55, 19, 69, 33, 48)(17, 53)(22, 63, 27, 61, 25, 62, 26, 58)\\&(24, 60)(34, 70).
\end{align*}
Applying the method explained in \cite{BW_J1} and \cite{BWJ} we  compute the desired Belyi map:
$$
f^2(X) = -\,2^{-4} \cdot 3^{-8} \cdot \frac{p(X)}{q(X)} \in \mathbb{Q}(X)
$$
where
\begin{align*}
p(X) = \;& \left( X^{12} + 8X^{11} - 10X^{10} - 40X^9 - 69X^8 - 96X^7 - 84X^6 \right.
\\
&\left. - 48X^5 - 21X^4 - 40X^3 - 26X^2 - 8X + 1 \right)^6,
\\
q(X) = \;& \left(X^3 + 3X + 2\right)^8 \left(X^4 + \frac{4}{3}X^3 - \frac{1}{3}\right)^8 \left(X^6 + \frac{3}{2}X^4 + \frac{1}{2}\right)^2 .
\end{align*}
This, obviously, gives us $f\in \mathbb{C}(X)$ ramified over $0$, $1$, $-1$, and $\infty$. Finally, we follow the approach in \cite{BWJ} to find a one-parameter family of polynomials with 
Galois group $PSp_6(2)$ over $\mathbb{Q}(t)$ corresponding to $\mathcal{H}$:
\begin{theorem*} 
Let $f(a,t,X) = p(a,X)-tq(a,X) \in \mathbb{Q}(a,t)[X]$
where 
\begin{align*}
p(a,X) =\; 
& \left( X^{12} + X^{11} + \left(144a + \frac{1}{8} \right)X^{10} + 40aX^9 + \left(-1728a^2 + \frac{21}{4}a\right)X^8 \right. \\
&+ \left(-576a^2 + \frac{3}{8}a\right)X^7 - 84a^2X^6 - 6a^2X^5 + \left(144a^3 - \frac{3}{64}a^2\right)X^4  \\
& \left. + \;40a^3X^3 + \frac{13}{4}a^3X^2 + \frac{1}{8}a^3X + a^4 \right)^3 ,
\end{align*}
and
\begin{align*}
q(a,X) =\; & \left( X^6 - 12aX^4 + \frac{1}{2}a^2 \right) \cdot \left(X^3 - 24aX - 2a \right)^4\\
& \cdot \left(X^4 + \frac{1}{6}X^3 + \frac{1}{24}a \right)^4.
\end{align*}
Then the Galois group of $f$ over $\mathbb{Q}(a,t)$ is isomorphic to $PSp_6(2)\xhookrightarrow{}S_{36}$, and the branch cycle structure of $f$ with respect to $t$ is given by $(3^{12}$, $1^{12}.2^{12}$, $1^{12}.2^{12}$, $1^6.2.4^7)$.
\end{theorem*}

The polynomial is also contained in the ancillary \texttt{Magma}-readable \cite{Magma} file.
Another polynomial with $PSp_6(2)$ as Galois group of degree $28$ (associated to a different class vector) that possesses infinitely many totally real specializations was found recently by the two authors and J. König, see \cite{BWJ}.

In order the prove the theorem we require the following observation:
\newpage
\begin{lemma*}
Let $K$ be an arbitrary field, $p(X),q(X)\in K[X]$ be coprime and $G:=\text{Gal}\,(p(X)-tq(X)\mid K(t))$. Further assume there exists an irreducible polynomial $r(X)\in K(t)[X]$ of degree $n$ such that $r(X)\in K(s)[X]$ becomes reducible where $t= \frac{p(s)}{q(s)}$. Then there exists a divisor $d\neq 1$ of $n$ such that $G$ posseses an index $d$ subgroup. 
\end{lemma*}

\begin{proof}
Let $L$ denote the splitting field of the irreducible polynomial $p(X)-tq(X)$ over $K(t)$, and $y$ be a root of $r(X)$ in a splitting field over $K(t)$:
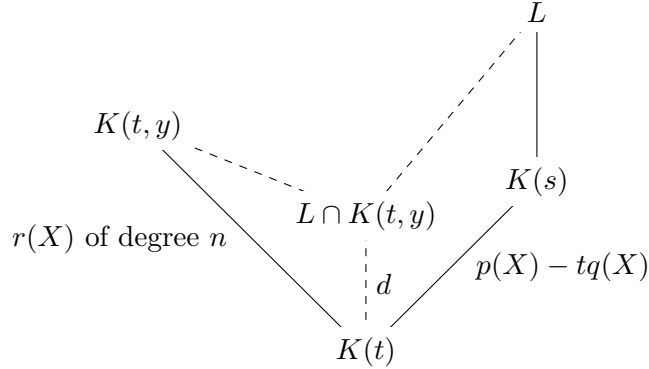
\begin{center}
\begin{tikzpicture}[scale = 1.5]

\draw (4.5,1.5) -- (4.5, 3);
\draw(3,0) -- (4.5,1.5);
\draw(3,0) -- (1,2);
\draw[dashed] (3,0) -- (3,1.2);
\draw[dashed] (3,1.2) -- (1,2);
\draw[dashed] (3,1.2) -- (4.5,3);

\node at (3,0.6) [right] {$d$};

\node (A) at (3,0) [fill = white] {$K(t)$};
\node (B) at (1,2) [fill = white] {$K(t,y)$};
\node (C) at (4.5,1.5) [fill = white] {$K(s)$};
\node (D) at (4.5,3) [fill = white] {$L$};
\node (D) at (3,1.2) [fill = white] {$L \cap K(t,y)$};

\node at (2,1) [left] {$r(X)$ of degree $n$\phantom{0}};
\node at (3.75,.75) [right] {\phantom{0}$p(X)-tq(X)$};
\end{tikzpicture} 
\end{center}
Thanks to the assumption that $r(X)$ splits nontrivially over $K(s)$, the polynomial $p(X)-tq(X)$ is reducible over $K(t,y)$\footnote{A well known result in field theory states: Let $K$ be a field, and $a,b$ algebraic over $K$ with minimal polynomials $\mu_a,\mu_b\in K[X]$. Then $\mu_a$ is irreducible over $K(b)$ if and only if $\mu_b$ is irreducible over $K(a)$.}, thus $ K(t) \lvertneqq L\cap K(t,y)$. Via Galois correspondence $G= \text{Gal}(L\mid K(t))$ must contain an index $d$ subgroup where $d\neq 1$ is a divisor of $n$. 
\end{proof}

\begin{proof}[Proof of the theorem]
Let $f_1,p_1,q_1\in \mathbb{Q}(t)[X]$ denote the specializations of $f,p,q$ at the place $a \mapsto 1$, and $\bar{f}_1,\bar{p}_1,\bar{q}_1$ their images in $\mathbb{F}_{37}(t)[X]$ under the canonical projection.

Using $\texttt{Magma}$ the discriminant of $f$ turns out to be 
\begin{align*}
\Delta=\;& 2^{732} \cdot 3^{168} \cdot   \left(a-\frac{1}{512}\right)^{154} \cdot a^{290} \cdot t^{24} \\
&\cdot \left( t^2 + \left(-2592a - \frac{81}{16}\right)t + 1679616a^2 - 6561a + \frac{6561}{1024}\right) ^{12}.
\end{align*}
With this formula we see that $f$ and $f_1$ have exactly four branch points with respect to $t$. Furthermore the branch cycle structure of $f$ can be derived by inspecting the inseparability behaviour of $f$ evaluated at the places $t \mapsto 0$, $t \mapsto \infty$, and $t \mapsto r_i$ for $i=1,2$ where $r_1$ and $r_2$ denote the non-zero roots of $\Delta \in \mathbb{Q}(a)[t]$.

By a theorem of Malle (see \cite{Malle}), the two Galois groups $\text{Gal}(f\mid \mathbb{Q}(a,t))$ and $\text{Gal}(f_1\mid \mathbb{Q}(t))$ coincide. It remains to show that $\text{Gal}(f_1\mid \mathbb{Q}(t))$ is isomorphic to $PSp_6(2)$.

Since $\frac{1}{X-t}\cdot \bar{f}_1\left(\frac{\bar{p}_1(t)}{\bar{q}_1(t)},X\right)\in \mathbb{F}_{37}(t)[X]$ is irreducible the Galois group of $\bar{f}_1$ over $\mathbb{F}_{37}(t)$ must be 2-transitive of permutation degree 36, implying $\text{Gal}(\bar{f}_1\mid \mathbb{F}_{37}(t)) \in \{PSp_6(2),A_{36},S_{36}\}$. Dedekind reduction yields $\text{Gal}(f_1\mid \mathbb{Q}(t)) \in \{PSp_6(2),A_{36}, S_{36}\}$.
Since both discriminants of $f_1$ and $\bar{f}_1$ are squares, we can each exclude the group $S_{36}$. In particular, $\text{Gal}(f_1\mid \mathbb{Q}(t))$ is simple, and the corresponding function field extension must be regular, allowing us to apply a theorem of Beckmann, see \cite[Chapter I, Proposition 10.9]{MM}, to obtain $\text{Gal}(f_1\mid \mathbb{Q}(t)) \cong \text{Gal}(\bar{f_1}\mid \mathbb{F}_{37}(t))$.

Let $r(t,X)\in \mathbb{F}_{37}(t)[X]$ be the irreducible polynomial of degree 63 in the ancillary file, then $r\left( \frac{\bar{p}_1(t)}{\bar{q}_1(t)},X\right)$ becomes reducible over $\mathbb{F}_{37}(t)$\footnote{The polynomial $r$, which is a minimal polynomial of a primitive element of the fixed field of an index 63 subgroup of $PSp(6,2)$, was obtained by using the \texttt{Magma} command \texttt{GaloisSubgroup}.}. The previous lemma guarantees the existence of an index $d\neq 1$ subgroup of $\text{Gal}(\bar{f_1}\mid \mathbb{F}_{37}(t))$ where $d$ is a divisor of $63$. Since $A_{36}$ does not contain such a subgroup we end up with 
$\text{Gal}(\bar{f_1}\mid \mathbb{F}_{37}(t)) = PSp_6(2)$, completing the proof.
\end{proof}

\section*{Acknowledgements}

We would like to thank Peter Müller for some valuable suggestions regarding the proof of the theorem.

\end{document}